\documentclass[11pt]{article}
\usepackage{amsmath}
\usepackage{amsthm}
\usepackage{amssymb,amsfonts}
\usepackage{hyperref}
\usepackage{latexsym}
\usepackage{todonotes}
\usepackage{nicefrac}
\usepackage{epsfig}
\usepackage{stmaryrd}
\usepackage{setspace}
\usepackage{enumerate}
\usepackage[all]{xypic}
\usepackage{bbm,ifpdf,tikz}
\ifpdf
\usepackage{pdfsync}
\fi

\newtheorem{theorem}{Theorem}[section]

\newtheorem{lemma}[theorem]{Lemma}
\newtheorem{proposition}[theorem]{Proposition}
\newtheorem{corollary}[theorem]{Corollary}

\newtheorem{conjecture}[theorem]{Conjecture}

{
\theoremstyle{definition}

\newtheorem*{conjecture*}{Conjecture}

\newtheorem{remark}[theorem]{Remark}
}

\newenvironment{proofrecursion}{\noindent {\bf Proof of Theorem \ref{recursion}:}}{\qed \par\vspace{\baselineskip}}

\newcommand{\becircled}{\mathaccent "7017}
\newcommand{\excise}[1]{}
\newcommand{\Spec}{\operatorname{Spec}}

\newcommand{\id}{\operatorname{id}}

\renewcommand{\dim}{\operatorname{dim}}

\newcommand{\Sym}{\operatorname{Sym}}

\renewcommand{\and}{\qquad\text{and}\qquad}
\newcommand{\Ind}{\operatorname{Ind}}
\newcommand{\Res}{\operatorname{Res}}
\newcommand{\Hom}{\operatorname{Hom}}

\newcommand{\Z}{\mathbb{Z}}

\newcommand{\R}{\mathbb{R}}
\newcommand{\C}{\mathbb{C}}

\newcommand{\IH}{I\! H}

\newcommand{\cA}{\mathcal{A}}
\newcommand{\la}{\lambda}

\newcommand{\cs}{\C^\times}

\newcommand{\Conf}{\operatorname{Conf}}
\newcommand{\Con}{\operatorname{Conf}(n,\R^3)}

\newcommand{\Consm}{\operatorname{Conf}(n,G)/G}

\newcommand{\ch}{\operatorname{ch}}
\newcommand{\ag}{\Aut(\Gamma)}
\newcommand{\Aut}{\operatorname{Aut}}
\newcommand{\hG}{\hat{\Gamma}}

\newcommand{\startarray}{\begin{eqnarray*}}
\renewcommand{\endarray}{\end{eqnarray*}}

\begin{document}
\spacing{1.2}
\noindent{\Large\bf The Orlik-Terao algebra and the cohomology of configuration space}\\

\noindent{\bf Daniel Moseley}\\
Department of Mathematics, Jacksonville University, Jacksonville, FL 32211\\

\noindent{\bf Nicholas Proudfoot}\footnote{Supported by NSF grant DMS-0950383.}\\
Department of Mathematics, University of Oregon,
Eugene, OR 97403\\

\noindent{\bf Ben Young}\\
Department of Mathematics, University of Oregon,
Eugene, OR 97403\\
{\small
\begin{quote}
\noindent {\em Abstract.}
We give a recursive algorithm for computing the Orlik-Terao algebra of the Coxeter arrangement of type $A_{n-1}$
as a graded representation of $S_n$, and we give a conjectural description of this representation in terms of the cohomology
of the configuration space of $n$ points in $SU(2)$ modulo translation.  We also give a version of this conjecture
for more general graphical arrangements.
\end{quote} }

\section{Introduction}
We consider the subalgebra $OT_n$ of rational functions on $\C^n$ generated by
$\frac{1}{x_i-x_j}$ for all $i\neq j$.  This is a special case of a class of algebras called Orlik-Terao algebras,
which have received much recent attention \cite{Terao-OT,PS,
ST-OT,Schenck-OT,DR-OT,SSV,DGS-OT,Dinh-OT,Liu-OT,McBP,EPW}.
Our interest is in understanding $OT_n$ as a graded representation of the symmetric group $S_n$, which acts by
permuting the indices.

Let $C_n$ be the cohomology of the configuration space of $n$ labeled points in $\R^3$, which is also acted on by $S_n$.
The ring $C_n$ is related to $OT_n$ in two different ways.  The first is that $C_n$ is isomorphic to the quotient
of $OT_n$ by the ideal generated by the squares of the generators.  This can be seen explicitly by computing presentations
of the two rings, but there is also a much deeper geometric explanation.  
Braden and the second author proved that $OT_n$ is isomorphic to the equivariant
intersection cohomology of a certain hypertoric variety (Theorem \ref{bp}), 
and $C_n$ is isomorphic to the equivariant cohomology
of a certain smooth open subset of that hypertoric variety; the map from $OT_n$ to $C_n$ is simply the restriction
map in equivariant intersection cohomology.  By exploring this geometric relationship further and considering not only
the open subset in question but also other strata of higher codimension, we obtain a formula which allows us to recursively
compute $OT_n$ in terms of $C_n$ (Theorem \ref{recursion}).  Since the action of $S_n$ on $C_n$ is well understood,
this allows us to compute the action of $S_n$ on $OT_n$ for arbitrary $n$.

Once we do these computations, a different and {\em a priori} unrelated relationship between $OT_n$ and $C_n$
becomes apparent.  Let $R_n$ be the symmetric algebra of the irreducible permutation representation of $S_n$,
generated in degree two.
The ring $OT_n$ is naturally an algebra over $R_n$, and it is finitely generated and free as a graded module.
Thus we may define $M_n := OT_n\otimes_{R_n}\C$, and we have an $S_n$-equivariant isomorphism
$OT_n\cong R_n\otimes_\C M_n$.  This reduces the problem of understanding $OT_n$ to the problem of understanding $M_n$.
Let $D_n$ be the cohomology of the configuration space of $n$ labeled points in $SU(2)\cong S^3$ modulo the
action of $SU(2)$ by simultaneous left translation.  
It is easy to show that $C_n$ and $D_n$ are closely related; see Propositions \ref{cd} and \ref{cdw} for precise
statements.
Our computations suggest the following result, which is the main conjecture in this paper (Conjecture \ref{main}):
\begin{conjecture*}[] There exists an isomorphism of graded $S_n$ representations
$M_n \cong D_n.$
\end{conjecture*}

Given that we have descriptions of both $M_n$ and $D_n$ in terms of $C_n$, one would think that this conjecture
would be easy to prove.  However, our recursive formula for $M_n$ involves plethysms of symmetric functions,
and while plethysms are fine for computing in SAGE, it is notoriously difficult to use them to prove anything.
\\

Our paper is structured roughly in the reverse of the order in which it was presented above.  
We begin in Section \ref{sec:conjectures} by giving a detailed account of our main conjecture, without any
discussion of how to compute $OT_n$ and $M_n$.  We also generalize our conjecture to arbitrary graphs.  
In Section \ref{geometry}, we explain how to use the equivariant intersection cohomology
of hypertoric varieties to compute $OT_n$.  Our main result in this section is Theorem \ref{recursion},
but we also we also do some extra work to translate our recursive formula to the language of symmetric functions
(Proposition \ref{symmetric}), since this is the most convenient formulation for actually computing with SAGE.
All of the code that was used for this project is available at \url{https://github.com/benyoung/ot}.

\vspace{\baselineskip}
\noindent
{\em Acknowledgments:}
The authors are grateful to Nick Addington, Richard Stanley, Dan Petersen, and Ben Webster for their valuable suggestions
and conversations.
All computations were performed with the assistance of the computer algebra systems SAGE~\cite{sage} 
and Macaulay 2 \cite{M2}.

\section{Conjectures}\label{sec:conjectures}
We begin by introducing the main players in our paper:  the Orlik-Terao algebra $OT_n$ and its finite dimensional
quotient $M_n$ (Section \ref{sec-ot}), the cohomology rings $C_n$ and $D_n$ of two closely related configuration spaces (Section \ref{sec-bn}), and our main conjecture relating them (Section \ref{sec-main}).  
We also generalize our conjecture to arbitrary graphs (Section \ref{graphs}).

\subsection{The Orlik-Terao algebra}\label{sec-ot}
Fix a positive integer $n$, and let $OT_n$ be the subalgebra of rational functions on $\C^n$ generated by
the elements $e_{ij} := \frac{1}{x_i-x_j}$ for all $i\neq j$.  
This algebra is known as the {\bf Orlik-Terao algebra} of the Coxeter arrangement of type $A_{n-1}$.
It follows from \cite[Theorem 4]{PS} and \cite[Proposition 2.7]{ST-OT}
that the ideal of relations between these generators
is generated by $e_{ij} + e_{ji}$ for all $i,j$
and $e_{ij}e_{jk} + e_{jk}e_{ki} + e_{ki}e_{ij}$ for all distinct triples $i,j,k$.
We regard $OT_n$ as a graded ring with $\deg(e_{ij}) = 2$.  Our goal is to understand $OT_n$
as a graded representation of the symmetric group $S_n$, which acts by permuting the indices.

Let $R_n := \C[z_1,\ldots,z_n]/\langle z_1+\cdots+z_n\rangle$, with its natural $S_n$ action,
and graded by putting $\deg(z_i) = 2$.
Consider the $S_n$-equivariant graded algebra homomorphism $\varphi_n:R_n\to OT_n$ taking $z_i$ to $\sum_{j\neq i} e_{ij}$.
This gives $OT_n$ the structure of a graded module over $R_n$, and it is in fact a free module by
\cite[Proposition 7]{PS}.  In other words, if we define $M_n := OT_n \otimes_{R_n}\C$ to be 
the ring obtained by setting $\varphi_n(z_i)$ equal to zero for all $i$, then there exists an isomorphism of graded 
$R_n$-modules 
\begin{equation*}\label{free}
OT_n\cong R_n\otimes_\C M_n.\end{equation*}
This isomorphism is not canonical, and is not compatible with the ring structures on the two sides.
However, it is compatible with the action of $S_n$ on both sides.
Thus, we may reduce the problem of understanding $OT_n$ 
as a graded representation to the problem of understanding $M_n$.


\begin{remark}\label{PoinMn}
It is easy to describe $M_n$ as a graded vector space.
For any finite dimensional graded vector space $V$ concentrated in even degree, let $H(V, q) := \sum q^i \dim V_{2i}$,
where $V_{2i}$ is the degree $2i$ part of $V$.  Then $H(M_n, q)$ 
is equal to the $h$-polynomial of the broken circuit
complex associated with the Coxeter arrangement of type $A_{n-1}$ \cite[Proposition 7]{PS}, which is equal to 
$(1+q)(1+2q)\cdots(1+(n-2)q)$.
\end{remark}

The following proposition was proved
by the first author \cite[Theorem 3.10]{Mos-thesis}; it may also be deduced from \cite[Theorem 3.3.3]{CEF}.

\begin{proposition}\label{mn-rep-st}
The sequence $\{M_n\}$ of graded representations
of symmetric groups is representation stable.
\end{proposition}

\subsection{Two configuration spaces}\label{sec-bn}
Consider the configuration space $\Con$ be the configuration space of $n$ labeled points in $\R^3$, which admits an action of $S_n$
given by permuting the labels.  
Let $$C_n:= H^*(\Con; \C),$$ which is a graded representation of $S_n$.
The ring $C_n$ has a presentation closely related to that of $OT_n$; it is isomorphic to the quotient
of $OT_n$ by the ideal generated by $e_{ij}^2$ for all $i,j$ \cite[Chapter III, Lemma 7.7]{Cohen}.  This algebra is also known as
the {\bf Artinian Orlik-Terao algebra} of the Coxeter arrangement of type $A_{n-1}$.
The structure of $C_n$ as a graded representation of $S_n$ is complicated but well understood;
see Equation \eqref{cn}.

Next, let $G = SU(2) \cong S^3$, and  consider the configuration space $\Conf(n,G)/G$ of $n$ labeled points in $G$
up to simultaneous translation by left multiplication.
This space admits an action of $S_n$ by permuting the labels;
let $$D_n := H^*(\Consm; \C),$$ which is a graded representation of $S_n$.

\begin{proposition}\label{cd}
There exists an isomorphism
$$C_{n-1} \cong \Res_{S_{n-1}}^{S_n}\!(D_{n})$$
of graded representations of $S_{n-1}$.
\end{proposition}

\begin{proof}
We have a diffeomorphism $\Consm\cong\Conf(n-1,\R^3)$ given by using the action of $G$ to take
the $n^\text{th}$ point to the identity, leaving the remaining $n$ points in $G\smallsetminus\{\id\}\cong \R^3$.  This diffeomorphism is equivariant with respect to the action of 
$S_{n-1}\subset S_n$.
\end{proof}

\begin{remark}\label{PoinBn}
The polynomial $H(D_n,q) = H(C_{n-1}, q)$ is equal to the $f$-polynomial of the broken circuit
complex associated with the Coxeter arrangement of type $A_{n-2}$ \cite[Theorem 4.3]{OT-Comm}, which is equal to 
$(1+q)(1+2q)\cdots(1+(n-2)q)$.
\end{remark}

Let $W_n := R_n/\langle z_iz_j\rangle$ be the ring obtained by truncating $R_n$ to degree two.  As a graded representation
of $S_n$, $W_n$ is isomorphic to the 1-dimensional trivial representation in degree zero plus the irreducible permutation representation of dimension $n-1$ in degree two.

\begin{proposition}\label{cdw}
There exists an isomorphism
$$C_{n} \cong D_n\otimes_\C W_n$$
of graded representations of $S_{n}$.
\end{proposition}

\begin{proof}
Consider the projection $\Conf(n+1,G)/G \to \Conf(n,G)/G$ given by forgetting the $(n+1)^\text{st}$
point.  This is an $S_n$-equivariant fiber bundle with fiber diffeomorphic to the complement 
of $n$ points in $G$.  The base
is simply connected and both the base and the fiber have cohomology only in even degree, thus
the Leray-Serre spectral sequence degenerates and the cohomology of the total
space is isomorphic to the tensor product of the cohomology of the base and the cohomology of the fiber.
This yields the desired isomorphism.
\end{proof}

\begin{remark}
Proposition \ref{cd} tells us that, if we know how to compute $D_{n+1}$, we know how to compute $C_n$.
Conversely, since $W_n$ is not a zero divisor in the semiring of graded representations of $S_n$,
Proposition \ref{cdw} tells us that we can recover $D_n$ from $C_n$.  This is important because there exist
extremely explicit formulas for $C_n$ in the literature; see Equation \eqref{cn}.
\end{remark}

\begin{corollary}\label{dn-rep-st}
The sequence $\{D_n\}$ of graded representations
of symmetric groups is representation stable.
\end{corollary}

\begin{proof}
Representation stability of $\{C_n\}$ (or, more generally, for the cohomology of the configuration space of any manifold)
was proved by Church \cite[Theorem 1]{Church}.  For $\{W_n\}$, it is obvious.  Representation stability of $\{D_n\}$
then follows from Proposition \ref{cdw}.
\end{proof}

Given a graded representation $V$ of $S_n$, let $\overline{V}$ be the ungraded representation
obtained by forgetting the grading.  In this section, we describe $\overline{C}_n$ and $\overline{D}_n$.
Let $Z_n\subset S_n$ be the cyclic group.

\begin{proposition}
There exist isomorphisms
$$\overline{C}_n\cong \C[S_n]\and \overline{D}_n\cong \C[S_n/Z_n] \cong \Ind_{Z_n}^{S_n}(\operatorname{triv})$$
of representations of $S_n$.
\end{proposition}

\begin{proof}
The first isomorphism is well-known, but we quickly review one proof here because we will use a very similar
argument for the second isomorphism.  Consider the action of $U(1)$ on $\R^3\cong \R\oplus\C$ given by rotation
on the second factor, which induces an action of $U(1)$ on $\Con$.  Since $H^*(\Con; \C)$ is concentrated in even
degree, this action is equivariantly formal, meaning that the $U(1)$-equivariant cohomology of $\Con$ is a free
module over the equivariant cohomology of a point.  It follows that there is a natural filtration on the cohomology
of the fixed point set $\Con^{U(1)}$ whose associated graded is 
isomorphic to the cohomology of $\Con$ \cite[Corollary 2.6]{Moseley}.  
Since the action of $U(1)$ commutes with the action of $S_n$,
this isomorphism is $S_n$-equivariant \cite[Proposition 2.8]{Moseley}.
We have $$\Con^{U(1)} \cong \Conf(n,\R)\simeq S_n,$$
so $$H^*\!\left(\Con^{U(1)}; \C\right)\cong \C[S_n].$$
Passing to the associated graded does not change the isomorphism type of an (ungraded) representation of a finite
group, thus $\overline{C}_n\cong \C[S_n]$.

For the second isomorphism, we note that $U(1)$ acts on $\Conf(n,G)/G$ by {\em right} translation, commuting
with the action of $S_n$, with fixed point set 
$$\left(\Conf(n,G)/G\right)^{U(1)}\cong \Conf(n, U(1))/U(1)\simeq S_n/Z_n.$$
The second isomorphism follows by the same argument.
\end{proof}

\begin{remark}
The filtration of $H^*(\Conf(n,\R))\cong \C[S_n]$ 
whose associated graded is isomorphic to $C_n$ can be described very explicitly.  
First, note that $\Conf(n,\R)$ is a disjoint union of contractible pieces, so its cohomology
ring is simply the ring of locally constant functions.
A {\bf Heaviside function} $h_{ij}$ is a function that takes the value 1 on one side of a given hyperplane $\{x_i=x_j\}$
and 0 on the other side.  We define the $p^\text{th}$
filtered piece $F_p\,\C[S_n]$ to be the vector space of functions that can be expressed as polynomials of degree at
most $p$ in the Heaviside functions.  This filtration, was first studied by Varchenko and Gelfand \cite{VG},
coincides with the one arising from equivariant cohomology \cite[Remark 4.9]{Moseley}.

Similarly, we may define a {\bf cyclic Heaviside function} $h_{ijk}$ on $\Conf(n, U(1))/U(1)$ by 
specifying a cyclic ordering of the $i^\text{th}$, $j^\text{th}$ and $k^\text{th}$ points.
This is equal to the pullback of $h_{ij}$ from
$\Conf(n-1,\R)$ along the isomorphism from $\Conf(n, U(1))/U(1)$ to $\Conf(n-1,\R)$ given by using the action
of $U(1)$ to move the $k^\text{th}$ point to the origin.
Since we know that the filtration of $H^*(\Conf(n-1,\R); \C)$ arising from equivariant cohomology coincides
with the one induced by Heaviside functions, we may conclude that, for any fixed index $k$, the filtration of 
$H^*(\Conf(n, U(1))/U(1); \C)$ arising from equivariant cohomology coincides with the filtration generated by the cyclic 
Heaviside functions $\{h_{ijk}\mid 0\leq i<j\leq n\}$.  Since the filtration arising from equivariant cohomology is preserved by the action of $S_n$,
it must also coincide with the filtration generated by {\em all} cyclic Heaviside functions, 
where all three indices are allowed to vary.
\end{remark}

\subsection{The main conjecture}\label{sec-main}
Our main conjecture is as follows.

\begin{conjecture}\label{main}
There exists an isomorphism of graded $S_n$ representations $M_n \cong D_n$.
\end{conjecture}

\begin{remark}
Using the computational technique described in Section \ref{geometry} (specifically Proposition \ref{symmetric}), 
we have checked Conjecture \ref{main} on a computer up to $n=10$.
\end{remark}

\begin{remark}
Remarks \ref{PoinMn} and \ref{PoinBn}
tell us that Conjecture \ref{main} holds at the level of graded vector spaces.
\end{remark}

\begin{remark}\label{ztf}
Since $W_n$ is not a zero divisor in the semiring of graded representations of $S_n$,
Conjecture \ref{main} is equivalent to the statement that $M_n\otimes_\C W_n\cong D_n\otimes_\C W_n$.
Since $OT_n\cong R_n\otimes_\C M_n$, we have $$M_n\otimes_\C W_n \cong OT_n/\langle z_iz_j\rangle.$$
On the other hand, Proposition \ref{cdw} says that $$D_n\otimes_\C W_n\cong C_n \cong OT_n/\langle e_{ij}^2\rangle.$$
We know that $\C\{z_iz_j\}$ and $\C\{e_{ij}^2\}$ are both isomorphic to the symmetric square of the irreducible permutation
representation, thus Conjecture \ref{main} holds in degrees zero, two, and four for all values of $n$.
\end{remark}

\begin{remark}
Since $z_iz_j = -e_{ij}^2 + f_{ij}$, where $f_{ij}$ is a certain sum of square-free monomials, 
it is natural to consider the family of rings
$$A_n(t) := OT_n/\langle (1-t)e_{ij}^2 -tz_iz_j\rangle
= OT_n/\langle e_{ij}^2 -tf_{ij}\rangle
,$$
where $t\in\C$.
By Remark \ref{ztf}, $A_n(0)\cong D_n\otimes_\C W_n$ and $A_n(1)\cong M_n\otimes_\C W_n$.
There exists a nonempty Zariski open subset $U\subset \C$ such that the restriction of this family to $U$
is flat, which means that the graded $S_n$ representations $A_n(t)$ are isomorphic for all $t\in U$.
If $0,1\in U$, this would imply Conjecture \ref{main}.  Unfortunately, this is not the case.
For example, when $n=4$, computations in Macaulay 2 reveal that $U = \C\smallsetminus\{0,1,-\frac 1 2\}$.

Put differently, this means that {\em most} ideals in $OT_4$ that are generated by a copy of the symmetric square
of the permutation representation in degree four are strictly larger than both $\langle e_{ij}^2\rangle$
and $\langle z_iz_j\rangle$.  These two ideals are exceptional, and our conjecture (which is true when $n=4$) says that they are exceptional in the same way.
\end{remark}

\excise{
\begin{remark}
It is tempting to generalize Conjecture \ref{main} to other hyperplane arrangements.  Let $\cA$
be a collection of hyperplanes in a vector space $V$.  Let $OT_\cA$ be the Orlik-Terao algebra, which is defined
to be the subalgebra of rational functions on $V$ generated by the reciprocals of the linear forms generating
the hyperplanes in $\cA$.  The algebra $OT_\cA$ is free as a graded algebra over a certain polynomial ring $R_\cA$
\cite[Proposition 7]{PS}, and we may define $M_\cA := OT_\cA \otimes_{R_\cA}\C$.  Let $C_\cA$
be the Artinian Orlik-Terao algebra, which is defined to be the quotient of $OT_\cA$ by the squares of the generators.
One natural generalization of Conjecture \ref{main} would be that $M_\cA$ is a tensor
factor of $C_\cA$, equivariantly for the group of symmetries of $\cA$.  

Unfortunately, this more general conjecture fails
even at the level of graded vector spaces.  The polynomial $H(M_\cA, q)$ is equal to the $h$-polynomial of the broken
circuit complex associated with $\cA$ \cite[Proposition 7]{PS}, 
while $H(C_\cA, q)$ is equal to the $f$-polynomial of the same broken
circuit complex \cite[Theorem 4.3]{OT-Comm}.  
In general, the $h$-polynomial does not divide the $f$-polynomial.  For example, when $\cA$
is the Coxeter arrangement of type $B_2$, the $h$-polynomial is equal to $1+2q$ and the $f$-polynomial is equal
to $(1+q)(1+3q)$.
\end{remark}
}

\subsection{Generalizing to graphs}\label{graphs}
In this section we generalize some of our results and conjectures to graphs; the cases described above correspond
to the complete graph.

Let $\Gamma$ be a simple connected graph with vertex set $[n]$, and let $\ag\subset S_n$ be the group of automorphisms
of $\Gamma$.  Let $OT_\Gamma$ be the Orlik-Terao algebra of the hyperplane arrangement associated with $\Gamma$;
this is the subalgebra of rational functions on $\C^n$ generated by $\frac{1}{x_i-x_j}$ whenever $i$ and $j$ are connected by
an edge.  It is a graded representation of the automorphism group $\Aut(\Gamma)\subset S_n$, with the generators in degree two.
We again have a map from $R_n$ to $OT_\Gamma$ as before, and we let $$M_\Gamma := OT_\Gamma\otimes_{R_n}\C.$$
Then there exists
a graded $\ag$-equivariant isomorphism 
$$OT_\Gamma \cong R_n\otimes_\C M_\Gamma,$$
and
$H(M_\Gamma, q) = h_\Gamma(q)$, 
the $h$-polynomial of the corresponding broken circuit complex \cite[Proposition 7]{PS}.

For any space $X$, consider the space $\Conf(\Gamma,X)$ of maps from the vertices of $\Gamma$ to $X$ such that
adjacent vertices map to different points.  Let
$$C_\Gamma := H^*(\Conf(\Gamma,\R^3) \and D_\Gamma := H^*(\Conf(\Gamma, G)/G; \C),$$
both graded representations of $\Aut(\Gamma)$.
Let $\hG$ be the cone over $\Gamma$; this is the graph with vertex set $[n+1]$ such that the $(n+1)^\text{st}$
vertex is connected to all other vertices and the subgraph spanned by the remaining vertices is equal to $\Gamma$.
The following proposition is a straightforward generalization of Proposition \ref{cd}.

\begin{proposition}\label{cdg}
There exists an isomorphism $$C_\Gamma\cong \Res^{\Aut(\hG)}_{\ag}\left(D_{\hG}\right)$$
of graded representations of $\ag$.  
\end{proposition}

The following conjecture is a natural generalization
of Conjecture \ref{main}.

\begin{conjecture}\label{maing}
For any simple connected graph $\Gamma$, there exists an isomorphism $$M_{\Gamma}\cong D_{\Gamma}$$
of graded representations of $\Aut(\Gamma)$.
In particular, there exists an isomorphism
$$\Res^{\Aut(\hG)}_{\ag}\left(M_{\hG}\right)\cong C_\Gamma.$$
\end{conjecture}

\begin{remark}
We have $H(M_{\hG}, q) = h_{\hG}(q) = f_\Gamma(q) = H(C_{\Gamma}, q)$, thus the second part of 
Conjecture \ref{maing} holds at the level of graded vector spaces.
\end{remark}

\section{Computing \boldmath{$M_n$} via hypertoric geometry}\label{geometry}
In this section, we explain how to use the geometry of hypertoric varieties to compute $M_n$.

\subsection{Hypertoric varieties}\label{sec:htv}
Given any hyperplane arrangement $\cA$ defined over the rational numbers, one may define
a variety called a {\bf hypertoric variety}.  Rather than giving a general construction, we will instead
give a direct definition of the hypertoric variety $X_n$ associated with the (doubled) Coxeter arrangement of type $A_{n-1}$.
For a general definition, see \cite{Pr07}.

Let $K_n$ be the lattice of rank $n(n-1)$ with basis $\{y_{ij}\mid i\neq j\in[n]\}$.
Consider the map
$\pi:K_n\to \Z\{x_1,\ldots,x_n\}$ taking $y_{ij}$ to $x_i-x_j$, and let $L_n$ be the image of $\pi$.  
Consider  the polynomial ring in $2n(n-1)$ variables $$Q_n := \C[z_{ij},w_{ij}]_{i\neq j}.$$
This ring has a grading by $K_n^*$ defined by putting $\deg(z_{ij}) = y_{ij}^* = -\deg(w_{ij})$.
Let $Q_n^L$ denote the subring of $Q_n$ spanned by homogeneous elements whose degree lies in the sublattice $L_n^*\subset K^*$.
Consider the map $$\mu_n:\Sym K_n^\C\to Q_n^L$$
taking $y_{ij}$ to $z_{ij}w_{ij}$, and define $$P_n := Q^L_n/\langle \mu_n(y)\mid \pi(y)=0\rangle\and X_n:= \Spec P_n.$$
The variety $X_n$ is the hypertoric variety that will be the main object of our attention.
Let $$T_n := \Hom(L_n^*, \cs)$$ be the algebraic torus of dimension $n-1$
with character lattice $L_n^*$; the grading of $P_n$ by $L_n^*$
induces an action of $T_n$ on $X_n$.  We also have an action of the symmetric group $S_n$ on $X_n$ given
by permuting indices.  This action does not commute with the action of $T_n$, but rather defines
an action of the semidirect product $T_n\rtimes S_n$ on $X_n$, where $S_n$ acts on $T_n$ in the obvious way.
The variety $X_n$ and its various symmetries are important to us due to the following theorem
\cite[Corollary 4.5]{TP08} (see also \cite[Proposition 3.16]{McBP}).

\begin{theorem}\label{bp}
There exists a canonical isomorphism $$\IH^*_{T_n}(X_n; \C)\cong OT_n$$
between the $T_n$-equivariant intersection cohomology of $X_n$ and $OT_n$.
This isomorphism is compatible with the maps from $$H^*_{T_n}(*; \C)\cong \Sym(L_n^*)_\C\cong R_n.$$
In particular, this implies that 
$$\IH^*(X_n; \C)\cong M_n.$$  Furthermore, all of these isomorphisms are compatible with the natural
actions of the symmetric group $S_n$.
\end{theorem}

We next define a stratification of $X_n$, following the general construction in \cite[Section 2]{PW07}.  For each
partition $B_1\sqcup\cdots\sqcup B_\ell$ of the set $[n]$, consider the ideal
$$J_n^B := \langle z_{ij}, w_{ij} \mid \text{there exists an $r$ such that $i,j\in B_r$}\rangle\subset Q_n.$$
This ideal descends to an ideal in $P_n$, which cuts out a subvariety $X_n^B\subset X_n$.
We have $X_n^{B'}\subset X_n^B$ if and only if $B$ refines $B'$, and we define 
$$\becircled X_n^B := X_n^B\smallsetminus\bigcup_{\text{$B$ refines $B'$}} X_n^{B'}.$$
Then $$X_n = \bigsqcup_{B}\becircled X_n^B$$
is a $T_n$-equivariant stratification of $X_n$.  
For each partition $B$, consider the subtorus $$T_n^B := T_{|B_1|}\times\cdots\times T_{|B_\ell|}\subset T_n,$$
embedded in the natural way.  Then $T_n^B$ is the stabilizer of
every point in $\becircled X_n^B$ \cite[Remark 2.3]{PW07}, thus the torus $T_n/T_n^B$ acts freely on $\becircled X_n^B$.
The quotient space is not Hausdorff, but if we take the quotient of $\becircled X_n^B$ by the maximal compact subtorus
of $T_n/T_n^B$, we obtain a manifold homeomorphic to $\Conf(\ell,\mathbb{R}^3)$ \cite[Proposition 5.2]{PW07}.
Finally, the stratum $\becircled X_n^B$ has a normal slice 
that is $T_n^B$-equivariantly isomorphic to $X_{|B_1|}\times\cdots\times X_{|B_\ell|}$
\cite[Lemma 2.4]{PW07}.

\subsection{A geometric recursion}
Given any partition $B$ of $[n]$, let $S_B$ be the stabilizer of $B$.  
Letting $m_i$ be the number of parts of $B$ of size $i$, we may express $S_B$ as a product of wreath products:
$$S_B \;\;\cong\;\; \prod_{i=1}^n\; S_i\wr S_{m_i}.$$
Given any partition $\la$ of $n$, let $B(\la)$ be the partition of $[n]$ given by putting $B_1 = \{1,\ldots,\la_1\}$, $B_2 = \{\la_1+1,\ldots,\la_1+\la_2\}$,
and so on.  Let $S_\la := S_{B(\la)}\subset S_n$ be the stabilizer of the partition $B(\la)$,
and let $W_\la := \prod S_{m_i}\subset S_\la$.

We define a graded representation
$M_n^c$ of $S_n$ by putting $(M_n^c)_i := (M_n)_{4(n-1)-i}$.  Theorem \ref{bp} says that $M_n\cong\IH^*(X_n; \C)$,
and $4(n-1) = 2\dim_\C X_n = \dim_\R X_n$, thus we have $M_n^c\cong \IH^*_c(X_n; \C)$
by Poincar\'e duality.
We will use the geometry of the hypertoric variety $X_n$ to prove the following result.

\begin{theorem}\label{recursion}
For any positive integer $n$, there exists an isomorphism of graded $S_n$ representations
$$OT_n \;\;\cong\;\; \bigoplus_{\la\vdash n} \;\operatorname{Ind}_{S_\la}^{S_n}\left(
C_{\ell(\la)}
\otimes(M^c_{\la_1}\otimes R_{\la_1})\otimes\cdots\otimes (M^c_{\la_{\ell(\la)}}\otimes
R_{\la_{\ell(\la)}})\right).$$
Here the subgroup $W_\la\subset S_\la$ acts on $C_{\ell(\la)}$ via the embedding $W_\la\hookrightarrow S_{\sum m_i} = S_{\ell(\la)}$,
and it also permutes the remaining tensor factors of the same size.  In addition, each factor of the form $M^c_{\la_j}\otimes R_{\la_j}$
is acted on by a separate subgroup $S_{\la_j}\subset S_\la$.
\end{theorem}

\begin{remark}
We claim that Theorem \ref{recursion} provides a recursive means of computing $M_n$ for all $n\geq 2$.  To see this, we first
observe that, since $OT_n\cong R_n\otimes_\C M_n$, it is possible to recover $M_n$ from $OT_n$.
Moreover, since $M_n$ vanishes in degrees greater than $2(n-2)$, it is possible to recover $M_n$
from the truncation of $OT_n$ to degree $2(n-2)$.  If we try to use Theorem \ref{recursion} to compute $OT_n$ and $M_n$
in terms of $M_k$ for $k<n$,
we run into the problem that $M_n^c$ appears on the right-hand side of the isomorphism.  However, $M_n^c$ vanishes
in degrees less than $4(n-1)-2(n-2) = 2n$, therefore we can compute the truncation of $OT_n$ to degree $2(n-2)$
without knowing $M_n$, and we avoid any circularity.
\end{remark}

\begin{remark}
Theorem \ref{recursion} can be generalized to a recursive expression for $OT_\cA$ in terms 
$M^c_{\cA'}$ for various restrictions $\cA'$ of $\cA$ and $C_{\cA''}$ for various localizations $\cA''$ of $\cA$.
Taking $\cA$ to be a graphical arrangement, this means we may compute $OT_\Gamma$ in terms of $M^c_{\Gamma'}$
for various contractions $\Gamma'$ of $\Gamma$ and $C_{\Gamma''}$ for various subgraphs $\Gamma''$ of $\Gamma$.
\end{remark}

Let $IC_{X_n}$ be the $T_n$-equivariant intersection cohomology sheaf on $X_n$.
For each partition $B = B_1\sqcup\cdots\sqcup B_\ell$ of $[n]$, let $\iota_B:\becircled X_n^B\hookrightarrow X_n$ be the inclusion.
To prove Theorem \ref{recursion}, we first establish the following lemma.  

\begin{lemma}\label{stalks}
There exists an $S_B$-equivariant isomorphism of graded vector spaces
$$\mathbb{H}^*_{T_n}(\becircled X_n^B; \iota_B^! IC_{X_n}) \cong 
C_{\ell}
\otimes (M^c_{|B_1|}\otimes R_{|B_1|})\otimes\cdots\otimes 
(M^c_{|B_\ell|}\otimes R_{|B_\ell|}).$$
\end{lemma}

\begin{proof}
The cohomology of the complex $\iota_B^! IC_{X_n}$ is a $T_n$-equivariant local system on $\becircled X_n^B$
whose fiber at a point is the compactly supported cohomology of the stalk of $IC_{X_n}$
at that point.  This is the same as the compactly supported intersection cohomology of the normal slice
$X_{|B_1|}\times\cdots\times X_{|B_\ell|}$ to $\becircled X_n^B\subset X_n$.
Since the quotient of $\becircled X_n^B$ by the maximal compact subtorus of $T_n$
is homeomorphic to the simply connected space $\Conf(\ell, \R^3)$, this local system is trivial.
We therefore have a spectral sequence $E$ with
$$E_2^{p,q} = H^p_{T_n}(\becircled X_n^B; \C) \otimes \IH_c^q(X_{|B_1|}\times\cdots\times X_{|B_\ell|}; \C)$$
that converges to $\mathbb{H}^*_{T_n}(\becircled X_n^B; \iota_B^! IC_{X_n})$.  Since these cohomology groups are concentrated
in even degree, all differentials are zero, therefore
\startarray
E_\infty = E_2 &=& H^*_{T_n}(\becircled X_n^B; \C) \otimes \IH_c^*(X_{|B_1|}\times\cdots\times X_{|B_\ell|}; \C)\\
&\cong& H^*(\Conf(\ell, \R^3); \C) \otimes H_{T_n^B}^*(*; \C) \otimes \IH_c^*(X_{|B_1|}\times\cdots\times X_{|B_\ell|}; \C)\\
&\cong& C_\ell \otimes R_{|B_1|}\otimes\cdots\otimes R_{|B_\ell|} \otimes M^c_{|B_1|}\otimes\cdots\otimes M^c_{|B_\ell|}\\
&\cong& C_{\ell} \otimes (M^c_{|B_1|}\otimes R_{|B_1|})\otimes\cdots\otimes 
(M^c_{|B_\ell|}\otimes R_{|B_\ell|}).
\endarray
Since the category of graded representations of $S_B$ is semisimple, we have 
a (noncanonical) $S_B$-equivariant isomorphism of graded vector spaces
$\mathbb{H}^*_{T_n}(\becircled X_n^B; \iota_B^! IC_{X_n})\cong E_\infty$.
\end{proof}

\begin{proofrecursion}
There is a spectral sequence $E$ with $$E_1^{p,q}\;\;\; = \bigoplus_{\substack{B_1\sqcup\cdots\sqcup B_\ell = [n]\\ \ell = n-p}}
\mathbb{H}^{p+q}_{T_n}(\becircled X_n^B; \iota_B^! IC_{X_n})$$ that converges to $\IH_{T_n}^*(X_n; \C)$ \cite[Section 3.4]{BGS96}.
By Lemma \ref{stalks}, $E_1^{p,q} = 0$ unless $p+q$ is even, thus 
\begin{eqnarray*}
E_\infty\;\;=\;\; E_1 &\cong& \bigoplus_B\; \mathbb{H}^*_{T_n}(\becircled X_n^B; \iota_B^! IC_{X_n})\\
&\cong& \bigoplus_B\; 
C_{\ell}
\otimes (M^c_{|B_1|}\otimes R_{|B_1|})\otimes\cdots\otimes 
(M^c_{|B_\ell|}\otimes R_{|B_\ell|}).
\end{eqnarray*}
As a representation of $S_n$, this is isomorphic to 
$$\bigoplus_{\la\vdash n}\; \operatorname{Ind}_{S_\la}^{S_n}\left( 
C_{\ell(\la)}
\otimes(M^c_{\la_1}\otimes R_{\la_1})\otimes\cdots\otimes (M^c_{\la_{\ell(\la)}}\otimes
R_{\la_{\ell(\la)}})\right).$$
Since the category of graded representations of $S_n$ is semisimple, we have a (noncanonical)
$S_n$-equivariant isomorphism of graded vector spaces $\IH_{T_n}^*(X_n; \C)\cong E_\infty$.  The result now follows
from Theorem \ref{bp}.
\end{proofrecursion}


\subsection{Symmetric functions}
In order to implement the recursive formula in Theorem \ref{recursion} in SAGE, it is convenient to convert
everything to the language of symmetric functions.  Let $\Lambda$ be the ring of symmetric functions
in infinitely many variables with coefficients in the formal power series ring $\Z[[q]]$.  
If $V$ is a graded representation of $S_n$, concentrated in even degree, with finite dimensional graded parts,
then its {\bf graded Frobenius characteristic} $\ch V$ is an element of $\Lambda$ of symmetric degree $n$; the coefficient
of $q^i$ is equal to the usual Frobenius characteristic of $V_{2i}$.  The Frobenius characteristic map is an isomorphism
of vector spaces, thus it is sufficient to compute $\ch OT_n$ and $\ch M_n$ for each $n$.  More concretely, expressing
$M_n$ as an $\mathbb{N}[q]$-linear combination of irreducible representations is equivalent to expressing $\ch M_n$
as an $\mathbb{N}[q]$-linear combination of Schur functions.

We begin by analyzing a single summand from Theorem \ref{recursion}.
The first piece that we need to understand better is $C_{\ell(\la)}$, which is acted on by the subgroup $W_\la \subset S_\la$.
We want to decompose $C_{\ell(\la)}$ into irreducible representations for this subgroup:
$$
C_{\ell(\la)}
\;\cong \bigoplus_{\substack{(\nu_1,\ldots,\nu_n)\\ \nu_i\vdash m_i}} V_{\nu_1}\otimes\cdots\otimes V_{\nu_n}\otimes
U(\nu_1,\ldots,\nu_n),$$
where 
\startarray
U(\nu_1,\ldots,\nu_n) &:=& \Hom_{W_\la}\left(V_{\nu_1}\otimes\cdots\otimes V_{\nu_n}, C_{\ell(\la)}\right)\\
&\cong& \Hom_{S_n}\left(\Ind_{W_\la}^{S_n}\left(V_{\nu_1}\otimes\cdots\otimes V_{\nu_n}\right), C_{\ell(\la)}\right)
\endarray 
is the graded vector space that records the graded multiplicity of $V_{\nu_1}\otimes\cdots\otimes V_{\nu_n}$ in $C_{\ell(\la)}$.

Let $Y_\la$ denote the Young subgroup $\prod_{i=1}^n S_{im_i}$, so that we have $S_\la\subset Y_\la\subset S_n$.
We will break up our induction into two steps, first from $S_\la$ to $Y_\la$ and then from $Y_\la$ to $S_n$.
We have
\startarray
&& \operatorname{Ind}_{S_\la}^{S_n}\left( 
C_{\ell(\la)}
\otimes(M^c_{\la_1}\otimes R_{\la_1})\otimes\cdots\otimes (M^c_{\la_{\ell(\la)}}\otimes
R_{\la_{\ell(\la)}})\right)\\\\
&\cong& \Ind_{Y_\la}^{S_n}\operatorname{Ind}_{S_\la}^{Y_\la}\left( 
C_{\ell(\la)}
\otimes(M^c_{\la_1}\otimes R_{\la_1})\otimes\cdots\otimes (M^c_{\la_{\ell(\la)}}\otimes
R_{\la_{\ell(\la)}})\right)\\\\
&\cong& \bigoplus_{\substack{(\nu_1,\ldots,\nu_n)\\ \nu_i\vdash m_i}} U(\nu_1,\ldots,\nu_n) \otimes
\Ind_{Y_\la}^{S_n} \left(\bigotimes_{i=1}^n \Ind_{S_i\wr S_{m_i}}^{S_{im_i}}\big(V_{\nu_i}\otimes (M_i^c\otimes R_i)^{\otimes m_i} \big)\right).
\endarray
The graded Frobenius characteristic map has the following properties \cite[Sections I.7-8]{Macdonald}:
\begin{itemize}
\item $\ch V_\nu = s_{\nu}$ (irreducibles go to Schur functions)
\item if $S_n \curvearrowright V$ and $S_n\curvearrowright V'$, then $\ch(V\oplus V') = \ch V + \ch V'$
\item if $S_n \curvearrowright V$ and $S_n\curvearrowright V'$, then
$\ch (V\otimes V') = \ch V * \ch V'$ (internal or ``Kroneker" product)
\item if $S_n \curvearrowright V$ and $S_n\curvearrowright V'$, then
$H(\Hom_{S_n}\left(V, V'\right), q) = \langle \ch V, \ch V'\rangle$ (inner product)
\item if $S_i \curvearrowright V$ and $S_j\curvearrowright V'$, then
$\ch\Ind_{S_i\times S_j}^{S_{i+j}}\left(V\otimes V'\right) = \ch V \cdot \ch V'$ (ordinary product)
\item if $S_i \curvearrowright V$ and $S_j\curvearrowright V'$, then 
$\ch \Ind_{S_i\wr S_{j}}^{S_{ij}}\left(V'\otimes V^{\otimes j}\right) = \ch V'[\ch V]$ (plethysm).
\end{itemize}
The analysis that we have done in this section, combined with Theorem \ref{recursion},
gives us the following result.

\begin{proposition}\label{symmetric}
We have
$$\ch OT_n\;\; = \sum_{\substack{(\nu_1,\ldots,\nu_n)\\ \sum i|\nu_i|=n}}
\Big\langle s_{\nu_1}\cdots s_{\nu_n}, \ch C_{\sum|\nu_i|}\Big\rangle\;\;
\prod_{i=1}^n\;
s_{\nu_i}[\ch M_i^c * \ch R_i].
$$
\end{proposition}

Recall that $M_i^c$ is just $M_i$ ``backward", so $\ch M_i^c$ is obtained from $\ch M_i$ by replacing $q$ with $q^{-1}$
and multiplying by $q^{2(i-1)}$.
Thus, in order to use Proposition \ref{symmetric} to compute $\ch OT_n$ and $\ch M_n$ recursively, it remains
only to find explicit formulas for $\ch C_n$ and $\ch R_n$.
A formula for $C_n$ is given by Hersh and Reiner \cite[Theorem 2.7]{HershReiner},
based on the work of Sundaram and Welker \cite[Theorem 4.4(iii)]{SunWel}.
Let $\zeta_n$ be an irreducible 1-dimensional representation of the cyclic group $Z_n\subset S_n$ whose character
takes a generator of $Z_n$ to a primitive $n^\text{th}$ root of unity, and let $\ell_n := \ch \Ind_{Z_n}^{S_n}(\zeta_n)$.
Let $h_n$ denote the complete homogeneous symmetric function of degree $n$.  Then
\begin{equation}\label{cn}\ch C_n = \sum_{\la\vdash n}q^{\sum (i-1)m_i}\prod_{i=1}^n h_{m_i}[\ell_i].\end{equation}
The description of $\ch R_n$ can be found in \cite[Section 5.6]{GarsiaNotes}:
$$\ch R_n = (1-q) \sum_{\la\vdash n} s_{\la}(1,q,q^2,\ldots)s_\la.$$

\bibliography{./symplectic}

\newcommand{\etalchar}[1]{$^{#1}$}
\def\cprime{$'$}
\providecommand{\bysame}{\leavevmode\hbox to3em{\hrulefill}\thinspace}
\providecommand{\MR}{\relax\ifhmode\unskip\space\fi MR }
\providecommand{\MRhref}[2]{%
  \href{http://www.ams.org/mathscinet-getitem?mr=#1}{#2}
}
\providecommand{\href}[2]{#2}
\begin{thebibliography}{DGT14}

\bibitem[BGS96]{BGS96}
Alexander Beilinson, Victor Ginzburg, and Wolfgang Soergel, \emph{Koszul
  duality patterns in representation theory}, J. Amer. Math. Soc. \textbf{9}
  (1996), no.~2, 473--527.

\bibitem[BP09]{TP08}
Tom Braden and Nicholas Proudfoot, \emph{The hypertoric intersection cohomology
  ring}, Invent. Math. \textbf{177} (2009), no.~2, 337--379.

\bibitem[CEF15]{CEF}
Thomas Church, Jordan~S. Ellenberg, and Benson Farb, \emph{F{I}-modules and
  stability for representations of symmetric groups}, Duke Math. J.
  \textbf{164} (2015), no.~9, 1833--1910.

\bibitem[Chu12]{Church}
Thomas Church, \emph{Homological stability for configuration spaces of
  manifolds}, Invent. Math. \textbf{188} (2012), no.~2, 465--504.

\bibitem[CLM76]{Cohen}
Frederick~R. Cohen, Thomas~J. Lada, and J.~Peter May, \emph{The homology of
  iterated loop spaces}, Lecture Notes in Mathematics, Vol. 533,
  Springer-Verlag, Berlin-New York, 1976.

\bibitem[DGT14]{DGS-OT}
Graham Denham, Mehdi Garrousian, and {\c{S}}tefan~O. Toh{\v{a}}neanu,
  \emph{Modular decomposition of the {O}rlik-{T}erao algebra}, Ann. Comb.
  \textbf{18} (2014), no.~2, 289--312.

\bibitem[EPW]{EPW}
Ben Elias, Nicholas Proudfoot, and Max Wakefield, \emph{{The Kazhdan-Lusztig
  polynomial of a matroid}}, preprint.

\bibitem[GS]{M2}
Daniel~R. Grayson and Michael~E. Stillman, \emph{Macaulay2, a software system
  for research in algebraic geometry}, Available at
  \url{http://www.math.uiuc.edu/Macaulay2/}.

\bibitem[HR]{HershReiner}
Patricia Hersh and Vic Reiner, \emph{{Representation stability for cohomology
  of configuration spaces in $\mathbb{R}^d$}}, \textsf{arXiv:1505.04196}.

\bibitem[Le14]{Dinh-OT}
Dinh~Van Le, \emph{On the {G}orensteinness of broken circuit complexes and
  {O}rlik-{T}erao ideals}, J. Combin. Theory Ser. A \textbf{123} (2014),
  169--185.

\bibitem[Liu]{Liu-OT}
Ricky~Ini Liu, \emph{{On the commutative quotient of Fomin-Kirillov algebras}},
  \textsf{arXiv:1409.4872}.

\bibitem[Mac95]{Macdonald}
I.~G. Macdonald, \emph{Symmetric functions and {H}all polynomials}, second ed.,
  Oxford Mathematical Monographs, The Clarendon Press, Oxford University Press,
  New York, 1995, With contributions by A. Zelevinsky, Oxford Science
  Publications.

\bibitem[Mos]{Moseley}
Daniel Moseley, \emph{Equivariant cohomology and the {V}archenko-{G}elfand
  filtration}, \textsf{arXiv:1110.5369}.

\bibitem[Mos12]{Mos-thesis}
Daniel Moseley, \emph{Group actions on hyperplane arrangements}, ProQuest LLC,
  Ann Arbor, MI, 2012, Thesis (Ph.D.)--University of Oregon.

\bibitem[MP15]{McBP}
Michael McBreen and Nicholas Proudfoot, \emph{Intersection cohomology and
  quantum cohomology of conical symplectic resolutions}, Algebr. Geom.
  \textbf{2} (2015), no.~5, 623--641.

\bibitem[OT94]{OT-Comm}
Peter Orlik and Hiroaki Terao, \emph{Commutative algebras for arrangements},
  Nagoya Math. J. \textbf{134} (1994), 65--73.

\bibitem[Pro03]{GarsiaNotes}
Claudio Procesi, 2003, \texttt{http://garsia.math.yorku.ca/ghana03/chapt2.pdf}.

\bibitem[Pro08]{Pr07}
Nicholas Proudfoot, \emph{A survey of hypertoric geometry and topology}, Toric
  Topology, Contemp. Math., vol. 460, Amer. Math. Soc., Providence, RI, 2008,
  pp.~323--338.

\bibitem[PS06]{PS}
Nicholas Proudfoot and David Speyer, \emph{A broken circuit ring}, Beitr\"age
  Algebra Geom. \textbf{47} (2006), no.~1, 161--166.

\bibitem[PW07]{PW07}
Nicholas Proudfoot and Ben Webster, \emph{Intersection cohomology of hypertoric
  varieties}, J. Algebraic Geom. \textbf{16} (2007), no.~1, 39--63.

\bibitem[S{\etalchar{+}}15]{sage}
W.\thinspace{}A. Stein et~al., \emph{{S}age {M}athematics {S}oftware ({V}ersion
  6.8)}, The Sage Development Team, 2015, {\tt http://www.sagemath.org}.

\bibitem[Sch11]{Schenck-OT}
Hal Schenck, \emph{Resonance varieties via blowups of {$\Bbb P^2$} and
  scrolls}, Int. Math. Res. Not. IMRN (2011), no.~20, 4756--4778.

\bibitem[SSV13]{SSV}
Raman Sanyal, Bernd Sturmfels, and Cynthia Vinzant, \emph{The entropic
  discriminant}, Adv. Math. \textbf{244} (2013), 678--707.

\bibitem[ST09]{ST-OT}
Hal Schenck and {\c{S}}tefan~O. Toh{\v{a}}neanu, \emph{The {O}rlik-{T}erao
  algebra and 2-formality}, Math. Res. Lett. \textbf{16} (2009), no.~1,
  171--182.

\bibitem[SW97]{SunWel}
Sheila Sundaram and Volkmar Welker, \emph{Group actions on arrangements of
  linear subspaces and applications to configuration spaces}, Trans. Amer.
  Math. Soc. \textbf{349} (1997), no.~4, 1389--1420.

\bibitem[Ter02]{Terao-OT}
Hiroaki Terao, \emph{Algebras generated by reciprocals of linear forms}, J.
  Algebra \textbf{250} (2002), no.~2, 549--558.

\bibitem[VG87]{VG}
A.~N. Varchenko and I.~M. Gel{\cprime}fand, \emph{Heaviside functions of a
  configuration of hyperplanes}, Funktsional. Anal. i Prilozhen. \textbf{21}
  (1987), no.~4, 1--18, 96.

\bibitem[VLR13]{DR-OT}
Dinh Van~Le and Tim R{\"o}mer, \emph{Broken circuit complexes and hyperplane
  arrangements}, J. Algebraic Combin. \textbf{38} (2013), no.~4, 989--1016.

\end{thebibliography}
\bibliographystyle{amsalpha}

\end{document}